\documentclass[10pt]{amsart}

\usepackage[margin=1.25 in]{geometry} 
\usepackage[utf8]{inputenc}

\usepackage{graphicx}
\usepackage{url}
\usepackage[pagebackref=true,colorlinks]{hyperref}
\hypersetup{citecolor={rgb,256:red,0;green,30;blue,120},
linkcolor={rgb,256:red,;green,60;blue,170}}
\usepackage{mathptmx}     
\usepackage{amsmath,amsfonts,amsthm,amssymb,color}
\usepackage{mathtools}
\usepackage{tikz-cd}

\usepackage{xparse}
\usepackage{calrsfs}
\usepackage{cleveref}
\DeclareMathAlphabet{\pazocal}{OMS}{zplm}{m}{n}

\NewDocumentCommand{\tens}{e{_^}}{%
  \mathbin{\mathop{\otimes}\displaylimits
    \IfValueT{#1}{_{#1}}
    \IfValueT{#2}{^{#2}}
  }%
}

\newcommand{\F}{\mathbb{F}_p}
\newcommand{\lie}{\mathcal{L}}
\newcommand{\Lie}{\mathrm{Lie}^s}
\newcommand{\free}{\mathrm{Free}}

\newcommand{\R}{\bar{\pazocal{R}}}
\newcommand{\B}{\mathrm{Bar}_{\bullet}}
\newcommand{\g}{\mathfrak{g}}
\newcommand{\Q}{\bar{Q}}
\newcommand{\Mod}{\mathrm{Mod}}
\newcommand{\CE}{\mathrm{CE}}

\newcommand{\id}{\mathrm{id}}
\newcommand{\wt}{\mathrm{wt}}
\newcommand{\Z}{\mathbb{Z}}

\usepackage{todonotes}

\newtheorem{theorem}{Theorem}[section]

\newtheorem{proposition}[theorem]{Proposition}

\theoremstyle{definition}
\newtheorem{definition}[theorem]{Definition}
\newtheorem{remark}[theorem]{Remark}

\newtheorem{notation}[theorem]{Notation}

\title{Mod $p$ homology of unordered configuration spaces of surfaces}
\author{Matthew Chen }
\address{Wayzata High Shcool, Plymouth, MN 55446}
\email{chenmat001@isd284.com}
\author{Adela YiYu Zhang}
\address{Department of Mathematics, Massachusetts Institute of Technology, Cambridge MA 02139}
\email{adelayyz@mit.edu}

\begin{document}

\maketitle
\begin{abstract}
    We provide a short proof that the dimensions of the mod $p$ homology groups of the unordered configuration space $B_k(T)$ of $k$ points in a torus are the same as its Betti numbers for $p>2$ and $k\leq p$. Hence the integral homology has no $p$-power torsion. The same argument works for the punctured genus $g$ surface with $g>0$, thereby recovering a result of Brantner-Hahn-Knudsen via Lubin-Tate theory.
\end{abstract}

\section{Introduction}
The unordered configuration space of $k$ points in a manifold  $M$ is the orbit space $$B_k(M):=\mathrm{Conf}_k(M)_{\Sigma_k}=\{(x_1,\ldots,x_k)\in M^{\times k},x_i\neq x_j\mathrm{\ for\ } i\neq j\}/\Sigma_k.$$ The main new result of this paper concerns the odd primary homology of $B_k(T)$, where $T$ is a closed torus.
\begin{theorem}\label{main}
Let $p$ be an odd prime.
     The dimension of $H_i(B_k(T);\F)$ over $\F$ is given by the $i$th Betti number $\beta_i(B_k(T))$ for all $i$ and $k\leq p$. Hence the integral homology of $B_k(T)$ has no $p$-power torsion for $k\leq p$. 
\end{theorem}

The study of unordered configuration spaces dates back to as early as Segal \cite{segal} and McDuff \cite{mcduff}. The rational homology groups of these objects are relatively well understood in cases of interests via classical methods, see for instance \cite{BC}\cite{BCT}\cite{kriz}\cite{totaro}\cite{felix}. In contrast, the odd primary homology groups of unordered configuration spaces have remained mostly intractable. Classically, the only known cases are the following:
when
 $M=\mathbb{R}^n$  for $1\leq n\leq \infty$  where $\bigoplus_{k\geq 0}H_*( B_k(M);\F)$ is the mod $p$ homology of the free $\mathbb E_n$-algebra on  $\mathbb{S}$ \cite{Einfinity} \cite{CLM}\cite{bmms},
and when the dimension of $M$ is odd   where $\bigoplus_{k\geq 0}H_*( B_k(M);\F)$ depends only on the $\F$-module $H_*(M;\F)$ \cite{BCT}\cite{milgram}\cite{BCM}.

More recently, advances in the computation of the homology of unordered configuration spaces are made possible by a result of Knudsen \cite{ben}. For any manifold $M$ and spectrum $X$, we can consider the \textit{labeled} configuration spectrum $$B_k(M;X):=\Sigma^{\infty}_+\mathrm{Conf}_k(M)\tens_{\Sigma_k} X^{\tens k}.$$ In particular $\Sigma^\infty_+B_k(M)=B_k(M;\mathbb{S})$. 
Denote by $s\lie$ the monad associated to the  free spectral Lie algebra functor $\free^{s\lie}$. The $\infty$-category of spectral Lie algebras is cotensored in Spaces, and  $(-)^{M^+}$ denotes the cotensor with the one-point compactification of $M$ in this category. Using the machinery of factorization homology, Knudsen established the following equivalence.
\begin{theorem}\cite[Section 3.4]{ben} \label{knudsen}
Let $M$ be a parallelizable $n$-manifold and $X$ a spectrum of weight one. Then there is an equivalence of weighted spectra
\begin{equation}\label{barconstruction}
    \bigoplus_{k\geq 1} B_k(M;X)  \simeq \mid \B(\mathrm{id}, s\lie, \free^{s\lie}(\Sigma^{n}X) ^{M^+}) \mid.
\end{equation}
 The left hand side is weighted by the index $k$ and right hand side induced by the weight on $X$.
\end{theorem}

Using the bar spectral sequence with rational coefficients associated to the right hand side of (\ref{barconstruction}), Knudsen \cite{knudsen} provided a general formula for the Betti numbers of unordered configuration spaces. Building on Knudsen's work, Drummond-Cole and Knudsen \cite{dck} produced explicit formulae of the Betti numbers of unordered configuration spaces of surfaces.  In \cite{bhk},  Brantner, Hahn, and Knudsen studied Knudsen's spectral sequence  with coefficients in  Morava $E$-theory at an odd prime. They computed the weight $p$ part of the labeled configuration spaces in $\mathbb{R}^n$ and punctured genus $g$ surfaces $\Sigma_{g,1}$ for $ g\geq 1$ with coefficient in a sphere. By letting the height go to infinity, they deduced that:
\begin{theorem}\cite[Theorem 1.10]{bhk}\label{bhktheorem}
Let $p$ be an odd prime.
   The integral homology of $B_p(\Sigma_{g,1})$, $g\geq 1$ has no $p$-power torsion. 
\end{theorem}
The computation via Morava $E$-theory becomes rather convoluted if one would like to apply it to the closed genus one surface $T$, which is the only parallelizable closed surface. Following a similar approach, the second author   studied Knudsen's spectral sequence  with odd primary coefficients in \cite{zhang} and showed that $H_*(B_k(M);\F),\ k=2,3$ depends on the cohomology ring $H^*(M^+;\F)$ when $M$ is an even dimensional parallelizable manifold, which is in contrast to the case when $M$ is odd dimensional.

In this paper, we build  on the  work of the second author and attack the computation of $H_*(B_k(T);\F)$ directly. We prove \Cref{main} by identifying the higher differentials in the  odd primary Knudsen's spectral sequence 
$$E^2_{s,t}(k)=\pi_s\pi_t\big( \B\big(\mathrm{id}, s\lie, \free^{s\lie}(\Sigma^{n}\mathbb{S}) ^{T^+}\big)\tens \mathbb{F}_p \big)(k)\Rightarrow H_{s+t}(B_k(T);\mathbb{F}_p)$$ for $k\leq p$. The argument we use involves a comparison with the spectral sequence computing  $H_*(B_k(T);\mathbb{Q})$ studied by  Knudsen \cite{knudsen} and simple dimension counting. The exact same argument works for the punctured genus $g$ surface for $g>0$, thereby providing a more direct proof of \Cref{bhktheorem}.

\subsection{Outline} In \Cref{section2}, we review the construction of Chevalley-Eilenberg complex of a shifted graded Lie algebra over a field and the structure of the odd primary homology of spectral Lie algebras. Then we recall previous computations of the $E^2$-page of the odd primary Knudsen's spectral sequence. Comparing with the computation of the rational homology of $B_k(M)$ by Knudsen where $M=T$ or $\Sigma_{g,1}$, we deduce the main result of the paper \Cref{main} in \Cref{section3}.

\subsection{Acknowledgements.}
The first author wishes to thank the Research Science Institute program and the Department of Mathematics at MIT for the opportunity. The second author would like to thank Jeremy Hahn, Haynes Miller, and Andrew Senger for helpful conversations.
\subsection{Conventions}
Let $k$ be a field. A weighted graded $k$-module  $M$ is an $\mathbb{N}$-indexed collection of $\mathbb{Z}$-graded $k$-modules  $\{M(w)\}_{w\in\mathbb{N}}$. The weight grading of an element $x\in M(w)$ is $w$. Morphisms are weight preserving morphisms of graded $k$-modules.
Denote by $\Mod_{k}$  the category of weighted graded $k$-modules. We omit the adjectives weighted, graded from here on. The Day convolution $\tens$ makes $\Mod_{k}$ a symmetric monoidal category and the Koszul sign rule $x\tens y=(-1)^{|x||y|}y\tens x$  depends only on the internal grading. Denote by $\wt_i(M)$ the weight $i$ part of the $k$-module $M$.

\section{Preliminaries}\label{section2}
\subsection{The Chevalley-Eilenberg complex}

A \textit{shifted Lie algebra} $L$ over $k$ is a $k$-module equipped with a shifted Lie bracket $$[- ,-]:L_m\tens L_n\rightarrow L_{m+n-1}$$ that satisfies graded commutativity $[x,y]=(-1)^{|x||y|}[y,x]$, the graded Jacobi identity $$(-1)^{|x||z|}[x,[y,z]]+(-1)^{|y||x|}[y,[z,x]]+(-1)^{|z||y|}[z,[x,y]]=0,$$ and adds weight. When $p=3$ we further require that $[[x,x],x]=0$ for all $x\in L$. Denote by $\Lie_{k}$ the category of shifted weighted graded Lie algebras over $k$.

\begin{definition}\cite{CE}\cite{may}\label{CE}
  Suppose that $k$ has characteristics away from two. For a $\Lie_k$-algebra $L$ , let $L_{\text{even}}$ and $L_{\text{odd}}$ denote the elements in $L$ with even and odd degree, respectively. The \textit{Chevalley-Eilenberg} complex of $L$ is the chain complex  
    $$\CE(L;k) = (\Gamma^\bullet(L_{\text{even}})\otimes \Lambda^\bullet (L_{\text{odd}}), \partial),$$
    where $\Gamma^\bullet$ and $\Lambda^\bullet$ are respectively the graded shifted divided power and exterior algebra functor over $k$, and the differential $\partial$ on an element
    $\gamma_{k_1}(x_1) \cdots \gamma_{k_m}(x_m)\langle y_1,  \dots, y_n\rangle\in\Gamma^{\bullet}(L_{\text{even}})\otimes \Lambda^\bullet (L_{\text{odd}})$
    is 
    \begin{align*}
        &\sum_{1\le i <j \le m} \gamma_{k_1}(x_1)\cdots \gamma_{k_i-1}(x_i)\cdots \gamma_{k_j-1}(x_j) \cdots \gamma_{k_m}(x_m)\langle [x_i,x_j], y_1, \dots y_n\rangle \\
        + &\sum_{1\le i < j \le n}(-1)^{i+j-1}\gamma_{k_1}(x_1)\cdots \gamma_{k_m}(x_m) \langle [y_i,y_j], y_1, \dots, \widehat{y_i},\dots \widehat{y_j},\dots y_n\rangle \\
        + &\frac 12 \sum_{i=1}^m \gamma_{k_1}(x_1)\cdots \gamma_{k_i-2}(x_i)\cdots \gamma_{k_m}(x_m) \langle [x_i, x_i], y_1, \dots, y_n\rangle\\
        + &\sum_{i=1}^m\sum_{j=1}^n (-1)^{j-1}\gamma_1([x_i, y_j])\gamma_{k_1}(x_1)\cdots \gamma_{k_i-1}(x_i)\cdots \gamma_{k_m}(x_m) \langle y_1, \dots, \widehat{y_j}, \dots, y_n\rangle.
    \end{align*}
\end{definition}
Note that the differential $\partial$ preserves weights, so $H_{*,*}(\wt_k(\CE(L;k)))=\wt_k(H_{*,*}(\CE(L;k)))$.
The Chevellay-Eilenberg complex is useful in computing
the $\Lie_k$-algebra homology. 
\begin{theorem}\cite{may}\cite{priddy}
For $L$ a $\Lie_k$-algebra, its  $\Lie_k$-algebra homology is given by
    $$ H^{\Lie_k}_{*,*}(L):= \pi_{*,*}(\B(\mathrm{id}, \Lie_k,L)\oplus k)\cong H_{*,*}(\mathrm{CE}(L;k)).$$
\end{theorem}
\subsection{Operations on the odd primary homology of spectral Lie algebras}

Building on the work of Arone-Mahowald \cite{am} and Johnson \cite{johnson},
Ching \cite{ching} and Salvatore showed that the Goodwillie derivatives $\{\partial_n(\mathrm{Id})\}_n$ of the identity functor $\mathrm{Id}:\mathrm{Top}_*\rightarrow\mathrm{Top}_*$  form an operad  in Spectra that is Koszul dual to the non-unital $\mathbb{E}_\infty$-operad. This is called the spectral Lie operad and we denote it by $s\lie$.
Algebras over the operad $s\lie$ are called \textit{spectral Lie algebras}, with structure maps $\partial_k(\mathrm{Id})\tens_{h\Sigma_k}L^{\tens k}\rightarrow L$ for all $k\geq 1$.

In \cite{kjaer}, Kjaer studied the structure of the mod $p$ homology of spectral Lie algebras for $p>2$ following the approach of Behrens \cite{behrens} and Antol\'{i}n-Camarena \cite{omar} when $p=2$.
\begin{proposition}\cite[Definition 3.2]{kjaer}
    Let $L$ be a spectral Lie algebra. Then $H_*(L;\F)$ admits unary operations of weight $p$ $$\overline{ \beta^{\epsilon} Q^j}: H_*(L;\F)\rightarrow H_{*+2(p-1)i-\epsilon-1}(L;\F),\ x\mapsto\xi_*( \sigma^{-1}\beta^{\epsilon} Q^j(x)) $$ for $\epsilon\in\{0,1\}, j\in\mathbb{Z}$. Here $\xi:\partial_p(\mathrm{Id})\tens_{h\Sigma_p}L^{\tens p}\rightarrow L$ is the $p$th structure map of the spectral Lie algebra $L$, $\sigma^{-1}$ the desuspension isomorphism, and $\beta^{\epsilon}\Q^j$ a mod $p$ Dyer-Lashof operation.
\end{proposition}
It follows from the unstability of Dyer-Lashof operations that $\overline{ \beta^{\epsilon} Q^j}(x)=0$ if $j<\frac{|x|}{2}$. There is also a $\Lie_{\F}$-algebra structure on $H_*(L;\F)$, induced by the second structure map  $$\partial_2(\mathrm{Id})\tens_{h\Sigma_2} L^{\tens 2}\simeq \partial_2(\mathrm{Id})\tens L^{\tens 2}_{h\Sigma_2}\simeq \mathbb{S}^{-1}\tens L^{\tens 2}_{h\Sigma_2}\rightarrow L.$$

\begin{proposition}\cite[Proposition 3.7]{kjaer}\label{oddvanish}
    For $L$ a spectral Lie algebra, $[\overline{ \beta^{\epsilon} Q^j}(x), y]=0$ for any $\epsilon, j$ and $x,y\in H_*(L;\F)$ . 

\end{proposition}

Define a functor $\Lie_{\R}:\Mod_{\mathbb{F}_p}\rightarrow\Mod_{\F}$ as follows. For $M\in\Mod_{\F}$, let $A$ be an $\F$-basis for the free shifted Lie algebra $\free^{\Lie_{\F}}(M)$. The graded $\F$-module $\Lie_{\R}(M)$ has basis
$$\{\overline{ \beta_1^{\epsilon_1} Q^{j_1}}\cdots \overline{ \beta_k^{\epsilon_k} Q^{j_k}}| x,\ \  x\in A,  j_k\geq \frac{|x|}{2}, j_i\geq pj_{i+1}-\epsilon_{i+1}\forall i\}.$$
The $\Lie_{\F}$-structure on $\free^{\Lie_{\F}}$ can be extended to that on $\Lie_{\R}(M)$ via \Cref{oddvanish}.

\begin{theorem}\cite[Theorem 5.2]{kjaer}\label{kjaer}
For $X$ a spectrum. there is an isomorphism of $\Lie_{\F}$-algebras
$$\Lie_{\R}(H_*(X;\F))\rightarrow H_*(\free^{s\lie}(X);\F).$$
\end{theorem}
\subsection{Odd primary Knudsen's spectral sequence}
The odd primary Knudsen's spectral sequence was first investigated by the second author in \cite{zhang}.
Using the skeletal filtration of the geometric realization of the bar construction in \Cref{knudsen}, we obtain Knudsen's spectral sequence with mod $p$ coefficients
\begin{equation}\label{oddprimarysseq}
    E^2_{s,t}(k)=\pi_s\pi_t\big( \B\big(\mathrm{id}, s\lie, \free^{s\lie}(\Sigma^{n}X) ^{M^+}\big)\tens \mathbb{F}_p \big)(k)\Rightarrow H_{s+t}(B_k(M;X);\mathbb{F}_p).
\end{equation}

By repeatedly applying \Cref{kjaer}, we see that the $E^2$-page is the homotopy group of a simplicial $\F$-module $V_\bullet$ with $(\Lie_{\R})^{\circ s}(L)$ as the $s$th simplicial level, where $$L=H_*(\free^{s\lie}(\Sigma^{n}X) ^{M^+};\F)\cong \widetilde{H}^*(M^+;\F)\tens \Lie_{\R}(\Sigma^n H_*(X;\F)).$$
Then $L$ has a $\Lie_{\F}$-structure  given by  $$[y_1\tens x_1,  y_2\tens x_2]=(y_1\cup y_2)\tens[x_1,x_2]$$ \cite[Proposition 5.9]{bhk} and  the action of the unary operations is given by $$\overline{ \beta^{\epsilon} Q^{j}}(y\tens x)=y\tens\overline{ \beta^{\epsilon} Q^{j}}(x)$$ if there is no nonzero Steenrod operation on $H^*(M;\F)$  other than $Sq^0$ \cite[Proposition 4.4]{zhang}.

At the time of this work, there is no published result on the relations among the unary operations $\overline{ \beta^{\epsilon} Q^{j}}$.\footnote{Through private communication, we were informed that Nikolay Konovalov has forthcoming work computing the odd primary relations via Goodwillie calculus.}  Nonetheless, at weight $k<p^2$ it suffices to know how unary operations and $\Lie_{\F}$-brackets commute, which is established by \Cref{oddvanish}. The face maps on $\Lie_{\F}$-brackets are simply $\Lie_{\F}$-algebra structure maps. In \cite{zhang}, the second author computed the $E^2$-page of the spectral sequence (\ref{oddprimarysseq}) in weight $k\leq p$ in terms of $\Lie_{\F}$-algebra homology.

\begin{proposition}\cite[Proposition 6.5]{zhang}\label{oddE2}
Let $M$ be a parallelizable manifold of dimension $n$ and $X$ any spectrum. Set $$\g= \widetilde{H}^*(M^+;\F)\tens \Lie_{\F}(\Sigma^n H_*(X;\F))$$ with  $\Lie_{\F}$-structure given by  $[y_1\tens x_1,  y_2\tens x_2]=(y_1\cup y_2)\tens[x_1,x_2].$  
\begin{enumerate}
    \item For $k<p$, the weight $k$ part of the spectral sequence (\ref{oddprimarysseq})  has $E^2$-page given by $$E^2_{s,t}(k)\cong\wt_k H_{s,t}(\CE(\g;\F)).$$ 
    \item For $p\geq 5$, the weight $p$ part of the spectral sequence (\ref{oddprimarysseq}) has $E^2$-page given by $$E^2(p)_{*,*}\cong \wt_pH_{*,*}(\CE(\g;\F))\oplus \bigoplus_{y\in H, x\in B} \F\Big\{\overline{ \beta^{\epsilon} Q^{j}}|y\tens x, \ \frac{|x|-|y|}{2}\leq j<\frac{|x|}{2}\Big\},$$ where $H$ is an $\F$-basis of $\widetilde{H}^*(M^+;\F)$ and $B$ an $\F$-basis of $H_*(X;\F)$.
\end{enumerate}
\end{proposition}
\begin{remark}\cite[Remark 6.6]{zhang}\label{triplebracketp=3}
     When $p=3$, there has to be an identity $\overline{\beta^{\epsilon} Q^{j}}(x)=[[x,x],x]$ when $x$ is a class of degree $2j$ in the mod 3 homology of a spectral Lie algebra $L$. This can be seen by computing the weight 3 part of the spectral sequence (\ref{oddprimarysseq}) for $M=\mathbb{R}^n$, $X=\mathbb{S}^{2j}$ and comparing with the weight 3 part of the free $\mathbb{E}_n$-algebra on the generator $x$ given in \cite[III]{CLM}. 
     
     In other words, the mod 3 homology of a spectral Lie algebra should have the structure of an \textit{operadic} $\Lie_{\mathbb{F}_3}$-algebra, denoted by $\mathrm{Lie}^{s,\mathrm{op}}_{\mathbb{F}_3}$, which does not require $[[x,x],x]=0$. The underlying module of $\Lie_{\R}(M)$ is thus given as follows:  let $B$ be an $\mathbb{F}_3$-basis for the free $\mathrm{Lie}^{s,\mathrm{op}}_{\mathbb{F}_3}$-algebra on $M$. The graded $\mathbb{F}_3$-module $\Lie_{\R}(M)$ has basis the quotient of
$$\{\overline{ \beta_1^{\epsilon_1} Q^{j_1}}\cdots \overline{ \beta_k^{\epsilon_k} Q^{j_k}}| x,\ \  x\in B,  j_k\geq \frac{|x|}{2}, j_i\geq 3j_{i+1}-\epsilon_{i+1}\forall i\}$$ by the relation $\overline{\beta^{\epsilon} Q^{j}}(x)=[[x,x],x]$ for all $x\in M$ with even degree.
\end{remark}

\section{Mod $p$ homology of $B_k(T)$}\label{section3}
Let $M=\Sigma_g$ be a genus $g$ surface with $g\geq 1$. Its integral cohomology ring is
$$H^*(\Sigma_g;\Z) = \left\{ \begin{array}{lcl}
\Z\{d\} & *=0\\
 \Z\{a_i\oplus b_i,\ i=1,\ldots,g\} &  *=1 \\
\Z\{c\} & *=2\\
0&\mbox{otherwise}
\end{array}\right.$$
with cup product given by $a_i\cup b_i=c$ for all $i$, $d\cup y =y$ for all $y\in H^*(\Sigma_g;\Z)$, and zero otherwise. Hence $H^*(\Sigma_{g,1}^+;\mathbb{Z})$ has product structure given by $a_i\cup b_i=c$ for all $i$ and zero otherwise, where $\Sigma_{g,1}$ is the punctured genus $g$ surface.

Using the fact that Knudsen's spectral sequence with rational coefficients always collapses on the $E^2$-page \cite{ben}, Drummond-Cole and Knudsen  produced explicit formulae for the Betti numbers of $B_k(\Sigma_g)$ and $B_k(\Sigma_{g,1})$ for all $k$ and $g$.
\begin{notation}\label{Qbasis}
Denote by $\g$ the $\Lie_{\F}$-algebra $$\g= H^*(\Sigma_g;\F)\tens \Lie_{\F} (\F\{x_2\})$$ with $x_2$ in internal degree 2 and weight 1, with the $\Lie_{\F}$-structure  is given by $[y\tens x,  y'\tens x']=(y\cup y')\tens[x,x'].$ An $\F$-basis for $\g$ is $B=\{y\tens x_2, y \tens [x_2,x_2],\  y=a_1,b_1,\ldots,a_g,b_g,c,d\}$.

Let $\g'$ be $\Lie_{\mathbb{Q}}$-algebra $$H^*(\Sigma_g; \mathbb Q)\otimes \free^{\Lie_{\mathbb{Q}}}(\mathbb Q\{x_2\})$$ with brackets given by the same formula as above.
    A $\mathbb{Q}$-basis for $\mathfrak g'$  is also $B$.
\end{notation}

\begin{notation}\label{puncturedQbasis}
Denote by $\g_1$ the $\Lie_{\F}$-algebra $$\g_1= \widetilde H^*(\Sigma_{g,1}^+;\F)\tens \Lie_{\F} (\F\{x_2\})$$ with $x_2$ in internal degree 2 and weight 1, with the $\Lie_{\F}$-structure  is given by $[y\tens x,  y'\tens x']=(y\cup y')\tens[x,x'].$ An $\F$-basis for $\g$ is $B_1=\{y\tens x_2, y \tens [x_2,x_2],\  y=a_1,b_1,\ldots,a_g,b_g,c\}$.

Let $\g'_1$ be $\Lie_{\mathbb{Q}}$-algebra $$\widetilde H^*(\Sigma_g^+; \mathbb Q)\otimes \free^{\Lie_{\mathbb{Q}}}(\mathbb Q\{x_2\})$$ with brackets given by the same formula as above.
    A $\mathbb{Q}$-basis for $\mathfrak g'_1$  is also $B_1$.
\end{notation}

\begin{theorem}\cite{knudsen}\label{dck}
    The $i$th Betti number of $B_k(\Sigma_g)$ and $B_k(\Sigma_{g,1})$  are respectively equal to the dimension over $\mathbb{Q}$ of $\bigoplus_{s+t=i}H_{s,t}\big(\mathrm{wt}_k(\CE(\mathfrak g';\mathbb{Q}))\big)$ and $\bigoplus_{s+t=i}H_{s,t}\big(\mathrm{wt}_k(\CE(\mathfrak g'_1;\mathbb{Q}))\big)$ for all $i$. 
\end{theorem}
Explicit formulae for the Betti numbers $\beta_i(B_k(T))$ were obtained by  Drummond-Cole and Knudsen in \cite[Corollary 4.5-4.7]{dck} and $\beta_i(B_k(\Sigma_{g,1}))$  in \cite[Proposition 3.5]{dck}.

We will deduce the higher differentials in the spectral sequence
\begin{equation}\label{genusgsseq}
    E^2_{s,t}(k)=\pi_s\pi_t\big( \B\big(\mathrm{id}, s\lie, \free^{s\lie}(\Sigma^{n}\mathbb{S}) ^{M^+}\big)\tens \mathbb{F}_p \big)(k)\Rightarrow H_{s+t}(B_k(M);\mathbb{F}_p)
\end{equation}
where $M=T$ or $\Sigma_{g,1}$ by combining \Cref{oddE2} with \Cref{dck}.

\begin{theorem}[\Cref{main}]
    For $k\leq p$, the dimension of $H_i(B_k(T);\F)$ over $\F$ is equal to the Betti number $\beta_i(B_k(T))$ for all $i$. Hence the integral homology of $B_k(T)$ has no $p$-power torsion for $k\leq p$. 
\end{theorem}
\begin{proof}
 The isomorphism between the $\F$-basis of $\g$ and the $\mathbb{Q}$-basis of $\g'$ (cf. \Cref{Qbasis}) induces compatible isomorphisms between the basis $$\big\{\gamma_{k_1}(u_1) \cdots \gamma_{k_m}(u_m)\langle v_1,  \dots, v_n\rangle,\ u_1,\ldots,u_m,v_1,\ldots v_m\in B\big\}$$ on each simplicial level of $\CE(\g;\F)$ and $\CE(\g';\mathbb{Q})$. Since the integral cohomology of $T$ is torsion-free and $\gamma_p(y\tens x_2)$ in weight $p$ does not receive differentials,  the differentials in the two $\CE$ complexes preserve the isomorphism in weight $k\leq p$. Furthermore, when $p=3$ the simplicial objects $\B(\id, \Lie_{\mathbb{F}_3},\g)$ and $\B(\id, \mathrm{Lie}^{s,\mathrm{op}}_{\mathbb{F}_3},\g)$ are isomorphic in weight $k<p$. Hence
 \begin{equation}\label{CEsamedim}
     \dim_{\F}\bigoplus_{s+t=i}\wt_k H_{s,t}(\CE(\g;\F))=\dim_{\mathbb{Q}}\bigoplus_{s+t=i}\wt_kH_{s,t}(\CE(\g';\mathbb{Q}))=\beta_i(B_k(T))
 \end{equation}
for all $i$ and $k\leq p$, with the right equality given  by \Cref{dck}. 
 
 By \Cref{oddE2}.(1), the $E^2$-page of the weight $k<p$ part of the spectral sequence (\ref{genusgsseq}) is  $$E^2_{s,t}(k)\cong \wt_k H_{s,t}(\CE(\g;\F)).$$ Since $B_k(T)$ is of finite type, the dimension of $H_i(B_k(T);\F)$ over $\F$ is at least $\beta_i(B_k(T))$ for all $i$. Therefore no higher differential can happen by \Cref{CEsamedim}. 
 
Now we tackle the case $k=p$.
First we consider $p>3$.
By \Cref{oddE2}.(2),  the $E^2$-page of the weight $k$ part of the spectral sequence (\ref{genusgsseq}) is given by $$E^2_{*,*}(p)\cong\wt_p H_{*,*}(\CE(\g;\F))\oplus\F\{\overline{Q^0}|c\tens x_2, \overline{\beta Q^{0}}|c\tens x_2\}.$$ Note that the class $\overline{\beta Q^{0}}|c\tens x_2\in E^2_{1,-2}$ has total degree $-1$, so it has to killed by a class with total degree $0$ and simplicial degree at least 3. There is exactly one class of total degree $0$ in $\wt_p H_{*,*}(\CE(\g;\F))$ since $\beta_0(B_p(T))=1$, which is the class $\gamma_p(c\tens x_2)\in E^2_{p-1, 1-p}$. Therefore there has to be a $d_{p-2}$-differential $\gamma_p (c \tens x_2)\mapsto \overline{\beta Q^{0}}|c\tens x_2.$
Appealing  to the inequality $$\sum_i\dim_{\F} H_i(B_k(T);\F)\geq\sum_i\beta_i(B_k(T))\stackrel{(\ref{CEsamedim})}{=}\dim_{\F}\bigoplus_{s,t}\wt_p H_{s,t}(\CE(\g;\F))=\sum_{s,t}\dim_{\F}E^2_{s,t}(p)-2,$$ we deduce that no other higher differential can happen.

For $p=3$,  we deduce from \Cref{triplebracketp=3} that the normalized complexes of $\B(\id, \Lie_{\mathbb{F}_3},\g)$ and $\B(\id, \mathrm{Lie}^{s,\mathrm{op}}_{\mathbb{F}_3},\g)$ differ in the weight 3 part in that there is a differential $\gamma_3(c \tens x_2)\mapsto \overline{\beta Q^{0}}|c\tens x_2$ in the latter but not in the former. Hence in the spectral sequence (\ref{genusgsseq}), the element $$\gamma_3(c\tens x_2)=[[c\tens x_2,c\tens x_2],c\tens x_2]\in \Lie_{\R}\circ\Lie_{\R}\Big(H^*(T;\mathbb{F}_3)\tens \Lie_{\R} (\mathbb{F}_3\{x_2\})\Big)$$ with the two brackets come from different iterations of $\Lie_{\R}$ is mapped by the $d_1$-differential to $$[[c\tens x_2,c\tens x_2],c\tens x_2]\in \Lie_{\R}\Big(H^*(T;\mathbb{F}_3)\tens \Lie_{\R} (\mathbb{F}_3\{x_2\})\Big)$$ with both brackets coming from the same iteration of $\Lie_{\R}$. At weight 3, there are no more differentials between elements containing a unary operation and brackets. It follows that the $E^2$-page of the weight 3 part of the spectral sequence (\ref{genusgsseq}) has a basis given by the union of an $\mathbb{F}_3$-basis of $\wt_3 H_{*,*}(\CE(\g;\mathbb{F}_3))/\mathbb{F}_3\{\gamma_3(c\tens x_2)\}$ and $\{\overline{Q^0}|c\tens x_2\}$. Combining with \Cref{CEsamedim}, we see that $$\sum_i\beta_i(B_3(T))=\dim_{\mathbb{F}_3}\bigoplus_{s,t}\wt_3 H_{s,t}(\CE(\g;\mathbb{F}_3))=\sum_{s,t}\dim_{\mathbb{F}_3}E^2_{s,t}(3)\geq \sum_i\dim_{\mathbb{F}_3} H_i(B_3(T);\mathbb{F}_3).$$ Hence no higher differential can happen and equality is achieved.

Therefore the dimensions of $H_*(B_k(T);\F)$ over $\F$ agree with the Betti numbers for $k\leq p$. Since $B_k(T)$ is a finite complex, we further deduce that its integral homology has no $p$-power torsion for $k\leq p$.
\end{proof}
\begin{remark}
   The exact same argument works for the punctured genus $g$ surface by comparing the $\Lie$-algebra homology groups of $\g_1$ and $\g'_1$ (\Cref{puncturedQbasis})  weight $k\leq p$, thereby providing an elementary proof for \Cref{bhktheorem}.
\end{remark}

\begin{remark}
The $d_{p-2}$-differential $\gamma_p (c \tens x_2)\mapsto \overline{\beta Q^{0}}|c\tens x_2$ should be viewed as a universal differential, in the sense that it occurs  in  the universal case $M=\lim_{n\rightarrow\infty}\mathbb{R}^n$, cf.  \cite[Proposition 6.6]{zhang}. Then $\lim_{n\rightarrow\infty}\Omega^n\free^{s\lie}(\Sigma^{n}X)\simeq X$ and the spectral sequence (\ref{oddprimarysseq}) becomes $$E^2_{s,t}=\pi_s\pi_t \B(\mathrm{id}, s\lie, \mathbb{S}^r\tens \mathbb{F}_p )\Rightarrow \pi_{s+t}|\B(\id,s\lie,\mathbb{S}^r\tens \mathbb{F}_p)|\cong \pi_{s,t}(\free^{\mathbb{E}^{\mathrm{nu}}_\infty\tens \F}(\mathbb{S}^r)),$$ where $\mathbb{E}^{\mathrm{nu}}_\infty\tens \F$ is the non-unital $\mathbb{E}_\infty$-operad in the category of $\F$-module spectra. Heuristically, this is because the bottom non-vanishing mod $p$ Dyer-Lashof operation on a class $x$ of degree $2j$ in an $\mathbb{E}^{\mathrm{nu}}_\infty$-$\F$-algebra is given by $\Q^j(x)=x^{\tens p}$, so $\gamma_p(x)$ is redundant.
\end{remark}
\begin{remark}
    We expect the first $p$-power-torsion classes in $H_*(B_k(M);\Z)$ for $M=T, \Sigma_{g,1}$ to show up at weight $k=2p$ in the form of $\overline{\beta^\epsilon Q^1}|c\tens [x_2,x_2]\in H_*(B_{2p}(M);\F)$ with $\epsilon=0,1$.
\end{remark}

\end{document}